\documentclass[10pt]{article}
\usepackage{amsmath,amsthm,amssymb}
\usepackage{url}
\newtheorem{theorem}{Theorem}

\newtheorem{lemma}[theorem]{Lemma}
\newtheorem{proposition}[theorem]{Proposition}
\newtheorem{corollary}[theorem]{Corollary}

\theoremstyle{definition}
\newtheorem{definition}[theorem]{Definition}

\newtheorem{example}[theorem]{Example}

\begin{document}
\thispagestyle{empty}
\title{\textbf{\Large{Dimensions of semi-simple matrix algebras}}}
%
% Author names, with corresponding author information.
% [Update and move the \thanks{...} block as appropriate.]
%
\author{\textsc{Phillip Heikoop}
				\thanks{\textit{E-mail: ptheikoop@wpi.edu}}\\
%% In 10 years, I don't know that anyone's actually found me via one of these emails. But it might make sense to use a gmail or something permanent here.
\textit{\footnotesize{School of Mathematical Sciences}}\\
\textit{\footnotesize{Worcester Polytechnic Institute, MA 01609, USA}}\\
\textsc{Padraig \'O Cath\'ain}
				\thanks{\textit{E-mail: pocathain@wpi.edu}}\\
\textit{\footnotesize{School of Mathematical Sciences}}\\
\textit{\footnotesize{Worcester Polytechnic Institute, MA 01609, USA}}\\
}
\maketitle

\begin{abstract} 

For $n \geq 255$ we show that every integer of the form $n + 2m$ such that $0
\leq 2m \leq n^2 - \frac{9}{2}n\sqrt{n}$ is the dimension of a connected
semi-simple subalgebra of $\mathrm{M}_n(k)$, that is, a subalgebra isomorphic to a
direct sum of $t$ disjoint subalgebras $\mathrm{M}_{n_i}(k)$, where $\sum_{i = 1}^t n_i =
n$. From this, we conclude that the density of integers in $\left[0, \dots,
    n^2\right]$ which are dimensions of a semi-simple subalgebra of
    $\mathrm{M}_n(k)$ tends
to $1$ as $n \rightarrow \infty$.

\end{abstract}

\section{Introduction}

Students of mathematics are introduced at undergraduate level to a number of
distinct subjects called \textit{algebra}. In pre-calculus
algebra\footnote{P\'{O}C: when I first came to the US, I often answered the question \textit{What do you teach?} with the response \textit{Algebra}. On one memorable occasion, the response (from a non-mathematician) was \textit{Oh, are you not smart enough to teach calculus?}} one solves polynomial equations. In linear algebra, one learns to multiply matrices and studies systems of linear equations. In abstract algebra, a course on group theory often introduces symmetry and homomorphisms, and courses on rings, fields (and possibly) Galois theory finally return to the solution of polynomial equations.

In this paper we will attempt to pull together ideas from all of these algebraic subjects: we will introduce an (associative) \textit{algebra of matrices} as an algebraic structure related to both linear and abstract algebra. We will briefly review definitions for homomorphisms of algebras as motivation for the definition of simple and semi-simple algebras. Our main topic in this paper is an exploration of the dimensions of semi-simple algebras. After introducing and defining the necessary algebraic objects, our proof will turn out to be rather elementary: most of our time is spent on careful manipulation of inequalities and solving quadratic equations.

We have attempted to write this paper in a way which is accessible to anyone with some knowledge of linear algebra (where we assume knowledge of abstract vector spaces, and matrix operations in arbitrary dimension $n$) and abstract algebra (where we assume exposure to the definition of a ring and to the concept of a homomorphism). Further knowledge of groups, rings and fields would be helpful but is not essential. Everything described in Section \ref{Algebra} is well-known, we make no claim of originality in our coverage of this material. For a more detailed discussion of the structure theory of rings and algebras, we refer the reader to Chapter 5 of \textit{Groups and Representations} by Alperin and Bell \cite{AlperinBell}; or to Chapter 13 of \textit{Algebra: A graduate course} by Isaacs \cite{IsaacsAlgebra}.

\section{Algebraic preliminaries}\label{Algebra}

Throughout this paper, we will use the symbol $k$ for an arbitrary (commutative) field. The reader less comfortable with field theory will lose nothing by taking $k = \mathbb{R}$ or $k= \mathbb{C}$ throughout. In the study of linear algebra, the linear transformations of a vector space are of central importance. After fixing a basis for an $n$-dimensional vector space over $k$, we can identify linear transformations with \textit{matrices}. We denote by $\mathrm{M}_{n}(k)$ the set of all $n \times n$ matrices with entries in $k$. There are three natural operations on matrices over a field:
\begin{enumerate}
    \item Scalar multiplication: for any scalar $c \in k$ and any matrix $M \in \mathrm{M}_{n}(k)$, the matrix $cM$ is obtained by multiplying each element of $M$ by $c$.
    \item Matrix addition: for matrices $M_{1}, M_{2} \in \mathrm{M}_{n}(k)$, addition is defined entry-wise.
    \item Matrix multiplication has a slightly more complicated definition motivated by \textit{composition} of linear maps, which we do not record formally here.
\end{enumerate}
It is a routine exercise in a linear algebra course to verify that
$\textrm{M}_{n}(k)$ forms a \textit{vector space} under the operations of
scalar multiplication and matrix addition. In a course on ring theory, one
proves that $\textrm{M}_{n}(k)$ under matrix addition and matrix multiplication forms a \textit{ring}. This is essentially the definition of a $k$-algebra: it is both a ring and a vector space. There is also some compatibility condition: the zero element of the ring structure and the vector space structure must coincide, for example. The formal definition follows.

\begin{definition}
An (associative\footnote{There are also non-associative algebras, in which only the second and third conditions hold. Often, some alternative property is imposed: Lie algebras, in which the Jacobi identity is enforced, have a structure theory analogous to what we discuss in this section.}) \textit{algebra} over the field $k$ is a vector space $\mathcal{A}$ over $k$ together with a binary operation (multiplication) which satisfies the following conditions:
\begin{enumerate}
    \item Associativity: $A(BC) = (AB)C$ for all $A, B, C \in \mathcal{A}$.
    \item Distributivity: $A(B+C) = AB + AC$ and $(A+B)C = AC + BC$ for all $A, B, C \in \mathcal{A}$.
    \item Scalar compatibility: $(cA)(dB) = cdAB$ for all $c, d \in k$ and all $A, B \in \mathcal{A}$.
\end{enumerate}
\end{definition}
Any subset of $\mathrm{M}_{n}(k)$ closed under all three operations is a \textit{subalgebra}. The diagonal matrices and the upper triangular matrices in $\mathrm{M}_{n}(k)$ are both subalgebras of $\mathrm{M}_{n}(k)$, for example. We can also construct subalgebras as \textit{direct sums}: suppose that $\mathcal{A}$ is a subalgebra of $\mathrm{M}_{t}(k)$ and $\mathcal{B}$ is a subalgebra of $\mathrm{M}_{n-t}(k)$. Then
\[ \mathcal{A} \oplus \mathcal{B} = \left\{  \left(\begin{array}{ll} A & \textbf{0} \\ \textbf{0} & B  \end{array}\right) \mid A \in \mathcal{A}, B \in \mathcal{B} \right\} \]
is a subalgebra of $\mathrm{M}_{n}(k)$. To see this, observe that block-matrices can be added and multiplied just like ordinary matrices (though one must take some care since multiplication of elements is no longer commutative). In this section, we focus exclusively on subalgebras of $\mathrm{M}_{n}(k)$. This is justified by an analogue of Cayley's Theorem which states that every $n$-dimensional associative algebra over the field $k$ is isomorphic to a subalgebra of $\mathrm{M}_{n}(k)$, see for example \cite[Theorem 12.2]{Isaacs} (though note that Isaacs proves a more general result for rings).

The study of algebras over a field is a classical topic in abstract algebra. As in group theory, the organising principle is \textit{homomorphism}.

\begin{definition}
Let $\mathcal{A}$ and $\mathcal{B}$ be algebras over a field $k$. A function $\mathcal{A} \rightarrow \mathcal{B}$ is a \textit{homomorphism} (of $k$-algebras) if it preserves all the operations of the algebra. More precisely,
\begin{enumerate}
    \item $\phi(cA) = c \phi(A)$ for all $c \in k$ and $A \in \mathcal{A}$
    \item $\phi(A + B) = \phi(A) + \phi(B)$ for all $A, B \in \mathcal{A}$
    \item $\phi(AB) = \phi(A)\phi(B)$ for all $A, B \in \mathcal{A}$
\end{enumerate}
\end{definition}
The \textit{kernel} of a homomorphism $\phi: \mathcal{A} \rightarrow \mathcal{B}$ is
\[ \textrm{ker}(\phi) = \left\{ A \in \mathcal{A} \mid \phi(A) =
\textbf{0}_{\mathcal{B}}\right\}\,. \]
It can be verified that $\textrm{ker}(\phi)$ is closed under all three operations of $\mathcal{A}$ and also satisfies the stronger condition that $AN, NA \in \textrm{ker}(\phi)$ for any $A \in\mathcal{A}$ and $N \in \textrm{ker}(\phi)$. So $\textrm{ker}(\phi)$ is both a \textit{subspace} of $\mathcal{A}$ (because $\phi$ is a linear transformation) and a two-sided \textit{ideal} of $\mathcal{A}$ (because $\phi$ is a homomorphism of rings). The next example shows how to construct homomorphisms of algebras from fixed subspaces of the underlying vector space.

\begin{example}
Suppose that $V$ is an $n$-dimensional vector space with basis $\{e_{1},
e_{2}, \ldots, e_{n}\}$ and that $\mathcal{A}$ is a subalgebra of $\mathrm{M}_{n}(k)$ in which every element fixes the subspace $U$ generated by $\{e_{n-t+1}, e_{n-t+2}, \ldots, e_{n}\}$ of dimension $t$ setwise. With respect to this basis, every element of $\mathcal{A}$ can be written in block-upper-triangular form:
\[ A =  \left(\begin{array}{ll} A_{V/U} & A_{X} \\ \textbf{0} & A_{U}  \end{array}\right) \]
The function $\phi: A \rightarrow A_{U}$ is a homomorphism of algebras. The
image of $\phi$ is $\{ A_{U} \mid A \in \mathcal{A}\}$ which is a subalgebra
of $\mathrm{M}_{t}(k)$, while the kernel consists of all matrices for which $A_{U} = 0$. Unless $A_{X}$ is identically zero for all $A \in \mathcal{A}$, the algebra $\mathcal{A}$ is \textit{not} a direct sum.
\end{example}
The reader unfamiliar with homomorphisms of algebras is advised to construct fixed subspaces of the diagonal and upper-triangular matrices, and hence to describe homomorphisms from these algebras to smaller matrix algebras. In fact, there is a First Isomorphism Theorem for $k$-algebras (analogous to the result in group theory). We record it below.

\begin{theorem}\label{FIT}
Suppose that $\mathcal{I}$ is a two-sided ideal in the $k$-algebra $\mathcal{A}$. For any $A \in \mathcal{A}$ define the coset of $\mathcal{I}$ containing $A$ by
\[ A + \mathcal{I} = \{ A + N \mid N \in \mathcal{I} \}\,.\]
Scalar multiplication, vector addition and vector multiplication on
$\mathcal{A}/\mathcal{I} = \{ A + \mathcal{I} \mid A \in \mathcal{A}\}$ are
defined in the natural way\footnote{That is: $c(A + \mathcal{I}) = cA + \mathcal{I}$ for all $c \in k$ and $A \in \mathcal{A}$. Addition and multiplication are given by $(A + \mathcal{I}) + (B + \mathcal{I}) = (A+B) + \mathcal{I}$ and $(A+\mathcal{I})(B + \mathcal{I}) = AB+\mathcal{I}$ respectively.}. With these operations, $\mathcal{A}/\mathcal{I}$ is a $k$-algebra and the function $\phi_{\mathcal{I}}: A \rightarrow A + \mathcal{I}$ is a homomorphism with kernel $\mathcal{I}$.
\end{theorem}

Using Theorem \ref{FIT}, we can construct examples of homomorphisms not coming from fixed subspaces. One example comes from the set of \textit{strictly upper triangular} (s.u.t.) matrices, which are zero on and below the diagonal. Observe that the s.u.t. matrices form an algebra, and that this algebra is \textit{nilpotent}: there exists an integer $r$ such that all products of length $r$ in the algebra evaluate to $\textbf{0}$. It can be proved that the s.u.t. matrices form a two-sided ideal inside the upper triangular matrices. The quotient of $n\times n$ upper triangular matrices by s.u.t. matrices is isomorphic to the algebra of $n \times n$ diagonal matrices.

Of special interest in an algebraic theory are \textit{simple} objects: in this case $k$-algebras containing no non-trivial two-sided ideals. In contrast to group theory (where the classification of finite simple groups is one of the great achievements of mathematics) the classification of simple finite-dimensional $k$-algebras is relatively straightforward. We will not give a proof, but we will record the main result. Recall that a field $k$ is algebraically closed if every polynomial with coefficients in $k$ has a root in $k$.

\begin{theorem}[Wedderburn, Theorem 13.17 \cite{AlperinBell}]\label{Wed}
If the field $k$ is algebraically closed, then every simple finite-dimensional $k$-algebra is isomorphic to $\mathrm{M}_{n}(k)$ for some $n \in \mathbb{N}$.
\end{theorem}

Requiring $k$ to be algebraically closed means that we can always construct eigenvalues and eigenvectors of matrices. Without this assumption, there are other classes of simple algebras over $k$, which are fully described in the reference. Those additional algebras will not concern us in this paper (for any field $k$, a semi-simple sub-algebra of $\mathrm{M}_{n}(k)$ has dimension equal to one of the examples given in Theorem \ref{Wed}). A famous theorem of Jacobson gives a decomposition of an arbitrary algebra into nilpotent and simple components, greatly generalising our example of upper triangular and s.u.t. matrices above.

\begin{theorem}[Jacobson, Theorem 13.23 \cite{AlperinBell}]\label{Jac}
Let $\mathcal{A}$ be a finite-dimensional $k$-algebra. Then there exists a largest nilpotent ideal which contains all other nilpotent ideals, called the \textit{Jacobson radical}.
The quotient of $\mathcal{A}$ by the Jacobson radical is isomorphic to a direct sum of simple algebras.
\end{theorem}

Jacobson's theorem motivates the following definition.

\begin{definition}
An algebra is \textit{semi-simple} if its Jacobson radical is $(\textbf{0})$.
\end{definition}

Over an algebraically closed field, a semi-simple algebra is isomorphic to a direct sum of matrix algebras. A theorem of Malcev\footnote{This result is also called the Wedderburn Principal Theorem and the Wedderburn-Malcev theorem in the literature. In the context of Lie algebras, it is called the Jordan-Chevalley decomposition.} gives precise information about how a semi-simple algebra sits inside of $\mathrm{M}_{n}(k)$. We describe the general result: suppose that $k$ is algebraically closed and $\mathcal{A}$ is a subalgebra of $\mathrm{M}_{n}(k)$. Let $\mathcal{J}$ be the Jacobson radical of $\mathcal{A}$. Then there exists a subalgebra $\mathcal{S}$ of $\mathcal{A}$ which is isomorphic to the semi-simple quotient $\mathcal{A}/\mathcal{J}$, and disjoint from $\mathcal{J}$. Any other subalgebra which is a complement to $\mathcal{J}$ (as a vector space) is necessarily conjugate to $\mathcal{S}$. Up to conjugation in $\mathrm{M}_{n}(k)$, the algebra $\mathcal{S}$ is block-diagonal and $\mathcal{J}$ is s.u.t. and disjoint from $\mathcal{S}$. We record a special case of the general theory, for which we will have use in the next section.

\begin{theorem}[Malcev, \cite{Malcev}]\label{Malcev}
Let $k$ be an algebraically closed field, and let $\mathcal{A}$ be a semi-simple subalgebra of $\mathrm{M}_{n}(k)$.
If $\mathcal{A}$ is isomorphic to the direct sum $\bigoplus_{i=1}^{t}
\textrm{M}_{n_{i}}(k)$ then $\mathcal{A}$ is conjugate in
$\mathrm{M}_{n}(k)$ to an algebra of block diagonal matrices where the $i^{\textrm{th}}$ block has size $n_{i} \times n_{i}$.
\end{theorem}

If $\sum_{i=1}^{t} n_{i} < n$ then Malcev's theorem does not, in general, tell us what happens in the remainder of the matrix. The following algebras are semi-simple and isomorphic but not conjugate, for example:
\[\left\{ \left( \begin{array}{ll} A & \textbf{0} \\ \textbf{0} & A \end{array}\right) \mid A \in \mathrm{M}_{2}(k) \right\}, \,\,\, \left\{ \left( \begin{array}{ll} A & \textbf{0} \\ \textbf{0} & \textbf{0} \end{array}\right) \mid A \in \mathrm{M}_{2}(k) \right\}\,. \]
In the next section we will explore the possible dimensions for a semi-simple
subalgebra of $\textrm{M}_{n}(k)$.

\section{Dimensions of semi-simple subalgebras of $\textrm{M}_{n}(k)$}

A \textit{connected semi-simple subalgebra} (CSA) of $\mathrm{M}_{n}(k)$ is a subalgebra which is semisimple, such that the sum of the orders of the simple components is precisely $n$. (That is, $\sum_{i} n_{i} = n$ in the notation of Theorem \ref{Malcev}.) We will write $\mathcal{C}(n)$ for the set of dimensions of CSAs of $\mathrm{M}_{n}(k)$.

It follows directly from Theorem \ref{Wed} and the definition of a CSA that an integer $\ell \in [0, 1, \ldots, n^{2}]$ is the dimension of a CSA in $\mathrm{M}_{n}(k)$ if and only if there exists a partition $d_{1} + d_{2} + \ldots + d_{t}$ of $n$ such that $\ell = d_{1}^{2} + d_{2}^{2} + \ldots + d_{t}^{2}$. So we have the following explicit description:
\[ \mathcal{C}(n) = \{ \ell = d_{1}^{2} + d_{2}^{2} + \ldots + d_{t}^{2} \mid d_{1} + d_{2} + \ldots + d_{t} = n \} \,.\]
Since for any integer $t^{2} \equiv t \mod 2$, it follows that
\[ \ell =  d_{1}^{2} + d_{2}^{2} + \ldots + d_{t}^{2} \equiv d_{1} + d_{2} + \ldots + d_{t} = n \mod 2\,.\]
We record this result as a Lemma.
\begin{lemma} \label{parity}
If $\ell$ is the dimension of a CSA of $\mathrm{M}_{n}(k)$ then $\ell \equiv n \mod 2$.
\end{lemma}

Lemma \ref{parity} already shows that $\lim_{n\rightarrow \infty} n^{-2}|\mathcal{C}(n)| \leq 1/2$.
Unfortunately, the number of partitions of $n$is proportional to $e^{c\sqrt{n}}$ (see Chapter 15 of van Lint and Wilson's \textit{A course in combinatorics} for an introduction to the theory of partitions \cite{vanLintWilson}). So while elegant, this description of $\mathcal{C}(n)$ cannot be used in computations for even moderate values of $n$. Instead, we describe $\mathcal{C}(n)$ recursively.

\begin{lemma} \label{recursive}
Set $\mathcal{C}(0) = \{0\}$. For each $n \geq 2$,
\[ \mathcal{C}(n) = \cup_{j=1}^{n} \left\{ j^{2} + \ell \mid \ell \in \mathcal{C}(n-j)\right\}\,.\]
\end{lemma}

\begin{proof}
By Theorem \ref{Malcev}, we may assume that all CSAs are block-diagonal, and completely described by the corresponding partition of $n$.
Each CSA in $\mathrm{M}_{n}(k)$ has a full matrix algebra $\mathrm{M}_{j}(k)$ in its upper left corner, for some positive integer $j$.
In the lower right $(n-j) \times (n-j)$ block, there must be a semi-simple subalgebra of $\mathrm{M}_{n-j}(k)$, which has dimension in the set $\mathcal{C}(n-j)$.
This establishes the recursion. (Note that isomorphic subalgebras are counted multiple
times, and that there may be multiple non-isomorphic algebras with the same dimension.)
\end{proof}

This recursion is moderately efficient and allowed the authors to compute the sets $\mathcal{C}(n)$ for values of $n$
in the thousands without difficulty. By Lemma \ref{parity}, the dimension of a CSA is of the form $n + 2m$ for some integer $m$.

\begin{definition}
An even integer $2m$ is \textit{realisable} in dimension $n$ if there exists a CSA of
$\mathrm{M}_{n}(k)$ of dimension $n + 2m$. The \textit{width} of $2m$ is the minimal dimension in which it is realisable.
\end{definition}

If $2m$ is realisable in dimension $n$, then it is realisable in all larger dimensions. We adapt an argument of Savitt and Stanley \cite{SavittStanley} to give an upper bound on the width of $2m$. 
First observe that if $\sum_{i = 1}^d t_i^2 = n + 2m$ and $\sum_{i = 1}^d t_i
= n$ then $\sum_{i = 1}^d t_i\left(t_i - 1\right) = 2m$.
With this in mind we define a decomposition of an even integer $2m$.

\begin{definition}
For an even integer $2m$, define the \textit{greedy decomposition} as
\[ 2m = t_{0}(t_{0}-1) + t_{1}(t_{1}-1) + \ldots + t_{d}(t_{d}-1) \]
where all of the $t_{j}$ are positive integers and each $t_{j}$ is chosen to be maximal subject to the
condition $t_{j}(t_{j}-1) \leq n - \sum_{i=0}^{j-1} t_{i}(t_{i}-1)$.
The \textit{greedy width} of $2m$ is $\mathcal{G}(2m) = \sum_{j=1}^{d} t_{j}$.
\end{definition}

For example, the greedy decomposition of $40$ is
\[ 40 = 6(5) + 3(2) + 2(1) + 2(1)\]
and so $\mathcal{G}(40) = 13$. This implies that $\textrm{M}_{n}(k)$ contains a CSA of dimension $n+40$ for all $n \geq 13$. The greedy construction for an algebra of dimension $n+40$
is the direct sum
\[ \mathrm{M}_{6}(k) \oplus \mathrm{M}_{3}(k) \oplus \mathrm{M}_{2}(k) \oplus \mathrm{M}_{2}(k) \oplus \mathrm{M}_{1}(k) \times (n-13)\,, \]
where the last term indicates that one takes the direct sum of $n-13$ copies
of the one-dimensional matrix algebra. 
A little thought reveals that it is possible to do better: $\mathrm{M}_{10}(k)$
contains the sub-algebra $\mathrm{M}_{5}(k)\oplus \mathrm{M}_{5}(k)$ which has dimension $50 =
10 + 40$. Hence the actual width of $40$ is at most $10$.
Nevertheless, the greedy width is a useful upper bound on the minimal width of CSAs. Its behaviour is sufficiently regular that we can prove some theorems about it.

\begin{proposition} \label{greedy}
For every integer $m$, we have $\mathcal{G}(2m) \leq \mathrm{max} \left\{ \frac{3}{2}\sqrt{2m},\, 38\right\}$.
\end{proposition}

\begin{proof}
The claim may be verified computationally up to $m = 3042$. The last time that the inequality $\mathcal{G}(2m) \leq \frac{3}{2}\sqrt{2m}$ fails is when $2m = 640$.
In fact, the greedy decomposition of $640$ requires dimension $38$, with decomposition $25 + 6 + 3 + 2 + 2$, while $\frac{3}{2}\sqrt{2m} \sim 37.947$.

Assume that the claim holds for all even integers up to $2m-2$, where $m \geq 3042$. We will show by induction that the claim holds for $2m$.
The first term in the greedy expansion of $2m$ is the unique integer $t$ such that
\[ t(t-1) \leq 2m  < (t+1)t\,. \]
It follows that $2m - t(t-1) < (t+1)t - t(t-1) \leq 2t$. We apply the induction hypothesis:
\begin{eqnarray*}
\mathcal{G}(2m) & \leq & t + \mathcal{G}\left(2m- (t(t-1)\right) \\
 & \leq & t + \mathcal{G}\left(2t-2\right) \\
 & \leq & t + \textrm{max}\left(\frac{3}{2} \sqrt{2t-2},\, 38\right) \,.
\end{eqnarray*}

Since $t(t-1)\leq 2m$ by construction, we have $(t-1)^{2} \leq 2m$, equivalently $t-1 \leq \sqrt{2m}$.
We deal with each possibility separately: suppose first that $\mathcal{G}(2m) \leq t + 38$. Then:
\[ \mathcal{G}(2m) \leq t + 38 \leq \sqrt{2m} + 39 \leq \frac{3}{2} \sqrt{2m}\,,\]
which holds provided $m \geq 3042$. In the second case, again using $t -1 \leq \sqrt{2m}$,
\[ \mathcal{G}(2m) \leq t + \frac{3}{2}\sqrt{2t-2} \leq \sqrt{2m} + 1 + \frac{3}{2}\sqrt[4]{8m} \,.\]
We rearrange the inequality $1 + \sqrt{2m} + \frac{3}{2}\sqrt[4]{8m} \leq \frac{3}{2}\sqrt{2m}$ to get
$2 + 3\sqrt[4]{8m} \leq \sqrt{2m}$.
Multiplying both sides by $2$, we get
\[ 4 + 6 \sqrt[4]{8m} \leq \sqrt{8m} \,.\]
We make the substitution $y = \sqrt[4]{8m}$ and solve the resulting quadratic equation to find that the inequality holds when $y \geq 3 + \sqrt{13}$.
Solving in terms of $m$, we obtain $m\geq 119 + 33\sqrt{13}$ which implies that the result holds whenever $2m \geq 476$.
Since we already assumed that $m \geq 3042$, the result is established by induction.
\end{proof}

We apply Proposition \ref{greedy} to prove that for sufficiently large $n$, the dimensions of CSAs in the matrix algebra $\mathrm{M}_{n}(k)$ are structured and predictable, except for a (relatively) small region close to $n^{2}$.

\begin{theorem}\label{main}
For any $n \geq 225$, the algebra $\mathrm{M}_{n}(k)$ contains a connected semi-simple subalgebra of dimension $n + 2m$ for every even integer satisfying $0 \leq 2m \leq n^{2} - \frac{9}{2}n\sqrt{n}$.
\end{theorem}

\begin{proof}
Proposition \ref{greedy} shows that $\mathrm{M}_{n}(k)$ contains a CSA of dimension $n + 2m$ whenever we have $n \geq \max\{\frac{3}{2}\sqrt{2m}, 38\}$. For any $n \geq 38$, we invert this bound to find that $\mathrm{M}_{n}(k)$ contains CSAs of dimension $n + 2m$ for all $2m \leq \frac{4}{9} n^{2}$.

For each $1 \leq j \leq n-38$ we define the interval
\[ S_{j} = \left[ n + j(j-1), n + j(j-1) + \frac{4}{9} (n-j)^{2}\right]\,,\]
which is the set of dimensions of CSAs of $\textrm{M}_{n}(k)$ having a
$j\times j$ block in the upper left corner, and a CSA of width $(n-j)$ in the
lower right. 

Let us examine the behaviour of the function $f(j) = n + j(j-1) + \frac{4}{9} (n-j)^{2}$, where we continue to hold $n$ fixed. 
Simplifying we can write $f(j) = \frac{13}{9} j^{2} - \left( \frac{8}{9}n + 1\right) j + \frac{4}{9} n^{2} + n$. Taking the derivative with 
respect to $j$, we find that $f(j)$ reaches a local minimum when $j = \frac{4}{13}n - \frac{9}{26}$. So for values of $j$ in the interval 
$\left[ \frac{4}{13} n, n-38\right]$ the maximum entry of $S_{j}$ increases with $j$. 

Now we establish a sufficient condition on $j$ for the intervals $S_{j}$ and $S_{j+1}$ to overlap.
This occurs provided that the inequality $n + j(j-1) + \frac{4}{9} (n-j)^{2} \geq n + j(j+1)$ holds. 
To find the largest value of $j$ for which this occurs, we rearrange this condition as a quadratic in $j$:
\[ 4j^{2} - (8n+18)j + 4n^{2} \geq 0\,. \]
Solving using the quadratic formula, and simplifying the results, the roots of the polynomial are 
$n + \frac{9}{4} \pm \frac{3}{2} \sqrt{ n + \frac{27}{4}}$. The leading coefficient of the quadratic 
is positive, so the sets $S_{j}$ and $S_{j+1}$ intersect when $j$ is an integer in the interval 
$\left[0, n + \frac{9}{4} - \frac{3}{2} \sqrt{ n + \frac{27}{4}}\right]$. It will be convenient to 
use a slightly smaller upper bound for the interval of the form $j_{\textrm{max}} = n - \frac{3}{2}\sqrt{n}$ 
(a computation shows that this bound is valid when $n \geq \frac{9}{4}$).

We simplify the lower bound on elements of $S_{j_{\textrm{max}}}$ to get
\[ n + (n-\frac{3}{2}\sqrt{n})(n-\frac{3}{2}\sqrt{n}-1) = n^{2} - \frac{9}{2}n\sqrt{n} + \frac{9}{4} n - \frac{3}{2}\sqrt{n}\,. \]
(Note that taking the upper bound here would only improve our estimate by a linear term.) Again, it is convenient to replace this bound 
with the simpler form $n + n^{2} - \frac{9}{2} n\sqrt{n}$, expressing the upper bound on the dimension in the form $n + 2m$ yields our result.

To complete the proof, we need to enforce the conditions $j_{\textrm{max}} \leq n - 38$ and 
\[  n + j_{\textrm{max}}(j_{\textrm{max}}+1) + \frac{4}{9} (n-j_{\textrm{max}})^{2} \geq n + n^{2} - \frac{9}{2} n\sqrt{n}\]
simultaneously. Substituting for $j_{\textrm{max}}$, approximating $\frac{4}{9}(38)^{2}$ by $642$ and simplifying, we obtain the inequality 
\[ 75 n - 2048 \leq \frac{9}{2}n\sqrt{n}\,. \]
Making the substitution $y^{2} = n$, we obtain a cubic $9y^{3} - 250y^{2} + 4096 \geq 0$. Solving numerically, this polynomial inequality holds for all $y \geq 15$, and hence the result holds for $n \geq 225$.
\end{proof}

\section{An open problem}

One could carry out a similar analysis for not-necessarily-connected
subalgebras of $\textrm{M}_{n}(k)$. Taking unions of the dimensions of CSAs
established in Theorem \ref{main} of $M_{2}(k), \ldots, \textrm{M}_{n}(k)$ gives the following result.

\begin{corollary} \label{gencase}
For every $n \geq 49$, every integer in $[0, n^{2}-\frac{9}{2}n\sqrt{n} - 2n]$
is the dimension of a semi-simple subalgebra of $\mathrm{M}_{n}(k)$. As $n
\rightarrow \infty$ the proportion of integers in $[0, n^{2}]$ which are the
dimension of a semi-simple subalgebra of $\textrm{M}_{n}(k)$ tends to $1$.
\end{corollary}

\begin{proof}
All integers in the interval $[0, \ldots, n-1]$ may be realised by algebras of diagonal matrices. By Theorem \ref{main}, every integer with the same parity as $n$ in the interval $[n, n^{2}-\frac{9}{2}n\sqrt{n}]$ is the dimension of a semi-simple subalgebra of $\mathrm{M}_{n}(k)$. Similarly, every integer of the opposite parity to $n$ in the interval $[n-1, (n-1)^{2} - \frac{9}{2}(n-1)\sqrt{n-1}]$ is the dimension of a semi-simple subalgebra. Since $n^{2}-\frac{9}{2}n\sqrt{n}-2n \leq (n-1)^{2} - \frac{9}{2}(n-1)\sqrt{(n-1)}$ the result follows.
\end{proof}

We write $\textrm{gap}(n)$ for the first integer which is \textit{not} the dimension of a semi-simple subalgebra of $\mathrm{M}_{n}(k)$. 
We conjecture that there exist  constants $0 \leq \alpha \leq \beta \leq \frac{9}{2}$ such that
\[ \liminf_{n\rightarrow \infty} \frac{n^{2} - \textrm{gap}(n)}{n^{3/2}} = \alpha \]
and
\[ \limsup_{n\rightarrow \infty} \frac{n^{2} - \textrm{gap}(n)}{n^{3/2}} = \beta \,.\]
Computations up to $n = 600$ suggest that the function $n^{2} - \frac{13}{4}n\sqrt{n} - \textrm{gap}(n)$ 
is (mostly) positive for large $n$, and that $n^{2} - \frac{7}{2}n\sqrt{n} - \textrm{gap}(n)$ is (mostly) negative. 
We propose the problem of finding $\alpha$ and $\beta$ explicitly.

%\section*{Acknowledgements}
%This paper grew out of the first author's \textit{Major Qualifying Project}, completed at Worcester Polytechnic Institute under the direction of the second author.

\bibliographystyle{abbrv}
\flushleft{
\bibliography{Biblio2018}
}

\end{document}